\documentclass{article}
\usepackage[utf8]{inputenc}
\usepackage{amsthm,amssymb}
\usepackage[colorlinks,citecolor=blue,urlcolor=blue]{hyperref}
\usepackage{amsmath}

\usepackage{enumerate}
\usepackage{amsfonts}
\usepackage{graphicx}

\usepackage{hhline}
\usepackage[all,2cell,ps]{xy}

\newtheorem{theorem}{Theorem}[section]
\newtheorem{lemma}[theorem]{Lemma}
\newtheorem{proposition}[theorem]{Proposition}
\newtheorem{corollary}[theorem]{Corollary}

\newtheorem{definition}{Definition}[section]

\newtheorem{example}{Example}[section]
\newtheorem{remark}{Remark}[section]

\newcommand{\De}{\mathsf{De}}
\newcommand{\SRL}{\mathsf{SRL}}
\newcommand{\KSRL}{\mathsf{KSRL}}
\newcommand{\SRLs}{\mathsf{SRL}^*}
\newcommand{\Hs}{\mathsf{Hey}^*}
\newcommand{\NA}{\mathsf{NA}}
\newcommand{\SNA}{\mathsf{SNA}}
\newcommand{\ce}{\mathsf{c}}
\newcommand{\CK}{\mathsf{(CK)}}
\newcommand{\CC}{\mathsf{(C)}}
\newcommand{\SNAc}{\mathsf{SNA}^{c}}
\newcommand{\C}{\mathsf{C}}
\newcommand{\V}{\mathrm{V}}
\newcommand{\K}{\mathrm{K}}
\newcommand{\Con}{\mathrm{Con}}
\newcommand{\co}{\theta}
\newcommand{\ra}{\rightarrow}
\newcommand{\Ra}{\Rightarrow}
\newcommand{\we}{\wedge}
\newcommand{\sq}{\square}
\newcommand{\s}{\sim}
\newcommand{\oIF}{\mathrm{IF}^{o}}
\newcommand{\oX}{\mathrm{X}^{o}}

\title{Subresiduated Nelson Algebras}
\author{Noem\'i Lubomirsky, Paula Mench\'on, Hern\'an San Mart\'in}
\date{}

\begin{document}

\maketitle

\begin{abstract}
In this paper we generalize the well known relation between Heyting algebras and Nelson
algebras in the framework of subresiduated lattices. In order to make it possible,
we introduce the variety of subresiduated Nelson algebras. The
main tool for its study is the construction provided by Vakarelov. Using it, we characterize the
lattice of congruences of a subresiduated Nelson algebra through some of its implicative filters.
We use this characterization to describe simple and subdirectly irreducible algebras, as well as principal congruences.
Moreover, we prove that the variety of subresiduated Nelson algebras has
equationally definable principal congruences and also the congruence extension property.
Additionally, we present an equational base for the variety generated by the totally ordered subresiduated
Nelson algebras. Finally, we show that there exists an equivalence between the algebraic category
of subresiduated lattices and the algebraic category of centedred subresiduated Nelson algebras. 
\smallskip
  
  \noindent Keywords: Subresiduated lattices, Nelson algebras, twist construction, Kleene algebras.
\end{abstract}

\section{Introduction}

In this paper we study the convergence of ideas arising from different varieties of algebras
related to intuitionistic logics: Heyting algebras, subresiduated lattices and Nelson algebras.

Subresiduated lattices, which are a generalization of Heyting algebras, were introduced during the decade of 1970
by Epstein and Horn \cite{EH} as an algebraic counterpart of some logics with
strong implication previously studied by Lewy and Hacking \cite{H}. These logics are examples of 
subintuitionistic logics, i.e., logics in the language of intuitionistic logic that are defined semantically 
by using Kripke models, in the same way as intuitionistic logic is defined but without
requiring from the models some of the properties required in the intuitionistic case. Also in relation with the 
study of subintuitionistic logics, Celani and Jansana \cite{CJ} got these algebras as the elements of a subvariety
of the variety of weak Heyting algebras (see also \cite{CFMSM,CNSM}).
It is known that the variety $\mathsf{S4}$, whose members are the S4-algebras, is the algebraic semantics of the modal logic
$\mathbf{S4}$. This means that $\phi$ is a theorem of $\mathbf{S4}$ if and only if the variety $\mathsf{S4}$ satisfies $\phi\approx 1$.
The variety of subresiduated lattices corresponds to the variety of algebras defined for all the equations $\phi \approx 1$ satisfied in the variety $\mathsf{S4}$ where the only connectives that appear are conjunction
$\we$, disjunction $\vee$, bottom $\perp$, top $\top$ and a new connective of implication $\Ra$, called strict implication, defined by $\varphi \Ra \psi:= \square(\varphi\rightarrow\psi)$, where $\ra$ denotes the classical implication.

Nelson's constructive logic with strong negation, which was introduced in \cite{Nel} (see also \cite{R,S,V}), is a well-known and by now fairly well-understood non-classical logic that combines the constructive approach of positive intuitionistic logic with a classical (i.e. De Morgan) negation.
The algebraic models of this logic, forming a variety whose members are called Nelson algebras, have been studied since at least the late 1950's (firstly by Rasiowa; see \cite{R} and references therein) and are also by now a fairly well-understood class of algebras. One of the main algebraic insights on this variety came, towards the end of the 1970's, with the
realisation (independently due to Fidel and Vakarelov) that every Nelson algebra can be represented as a special binary product (here called a twist structure) of a Heyting algebra. 

The main goal of this manuscript is to extend the twist construction in the framework of subresiduated lattices, 
thus obtaining a new variety, whose members will be called subresiduated Nelson algebras.
More precisely, we will show that every subresiduated Nelson algebra can be represented as a twist structure of a subresiduated lattice. Another central objective of this paper is to study the subvariety of its totally ordered members. 

The paper is organized as follows. In Section \ref{s1} we recall the definition of Nelson algebra and also 
sketch the main constructions linking Heyting algebras with Nelson algebras. Moreover,
we recall the definition of a subresiduated lattice and some of its properties. 
In Section \ref{s2} we introduce subresiduated Nelson algebras, proving that the class of subresiduated Nelson algebras
(which is a variety) properly contains the variety of Nelson algebras.
We also show that every subresiduated Nelson algebra can be represented as a twist structure of a subresiduated lattice.
In Section \ref{s4} we prove that given an arbitrary subresiduated Nelson algebra, there exists an order isomorphism
between the lattice of its congruences and the lattice of its open implicative filters (which are a kind of implicative filters).
We use it in order to give a characterization of the principal congruences.
In particular, the mentioned characterization proves that the variety of subresiduated Nelson algebras has equationally definable principal congruences (EDPC). 
We also give a description of the simple and subdirectly irreducible algebras, 
and we prove that the variety of subresiduated Nelson algebras has the congruence extention property (CEP).
In Section \ref{s5} we study the class of totally ordered subresiduated Nelson algebras
in order to give an equational base for the class generated by this variety.
Finally, in Section \ref{s6} we characterize the subresiduated Nelson algebras that 
can be represented as a twist structure of a subresiduated lattice and we also prove that there exists 
an equivalence between the algebraic category
of subresiduated lattices and the algebraic category of centered subresiduated Nelson algebras, where
a centered subresiduated Nelson algebra is a subresiduated Nelson algebra endowed with a center, i.e.,
a fixed element with respect to the involution (this element is necessarily unique).

\section{Basic results}\label{s1}

In this section we recall the definition of Nelson algebra \cite{BC,C,Vig} and some links between Heyting algebras and Nelson algebras \cite{V}, as well as the definition of subresiduated lattice and some of its properties \cite{EH}.

A Kleene algebra \cite{BC,C,K} is a bounded distributive lattice endowed with a unary operation $\sim$
which satisfies the following identities: 
\begin{enumerate}[\normalfont Ne1)]
\item $\sim \sim x = x$,
\item $\sim (x\we y) = \sim x \vee \sim y$,
\item $(x\we \sim x)\we (y \vee \sim y) = x\we \sim x$.
\end{enumerate}

\begin{definition}
An algebra $\langle T,\we,\vee,\ra,\sim,0,1 \rangle$ of type $(2,2,2,1,0,0)$ is called a Nelson algebra if
$\langle T,\we,\vee,\sim,0,1 \rangle$ is a Kleene algebra and the following identities are satisfied:
\begin{enumerate}[\normalfont Ne4)]
\item $x\ra x = 1$,
\end{enumerate}
\begin{enumerate}[\normalfont Ne5)]
\item $x\ra (y\ra z) = (x\we y)\ra z$,
\end{enumerate}
\begin{enumerate}[\normalfont Ne6)]
\item $x\we (x\ra y) = x \we (\sim x \vee y)$,
\end{enumerate}
\begin{enumerate}[\normalfont Ne7)]
\item $x\ra y \leq \sim x\vee y$, 
\end{enumerate}
\begin{enumerate}[\normalfont Ne8)]
\item $x\ra (y\we z) = (x\ra y)\we (x\ra z)$.
\end{enumerate}
\end{definition}

If $\langle T,\we,\vee,\ra,\sim,0,1\rangle$ is an algebra of type $(2,2,2,1,0,0)$ where
$\langle T,\we,\vee,0,1 \rangle$ is a bounded distributive lattice
and the condition Ne6) is satisfied, then conditions Ne7) and Ne8) are also satisfied \cite{Mon1,Mon2}. 
We write $\NA$ for the variety of Nelson algebras.

There are two key constructions that relate Heyting algebras and Nelson algebras.
Given a Heyting algebra $A$, we define the set 
\begin{equation} \label{ka}
\K(A) = \{(a,b)\in A\times A: a\we b = 0\}
\end{equation}
and then endow it with the following operations:
\begin{itemize}
\item $(a,b)\we (c,d) = (a\we c,b\vee d)$,
\item $(a,b)\vee (c,d) = (a\vee c, b\we d)$,
\item $\s (a,b) = (b,a)$,
\item $(a,b)\Ra (c,d) = (a\ra c,a\we d)$,
\item $\perp = (0,1)$, 
\item $\top = (1,0)$.
\end{itemize}
Then $\langle \K(A),\we,\vee,\Ra,\perp,\top \rangle\in \NA$ \cite{V}. In the same manuscript,
Vakarelov proves that if $T\in \NA$, then the relation $\theta$ 
defined by 
\begin{equation}\label{rel}
x \theta y\;\text{if and only if}\; x\ra y = 1\;\text{and}\; y\ra x = 1 
\end{equation}
is an equivalence relation such that $\langle T/\theta,\we,\vee,\ra,0,1 \rangle$ is a Heyting algebra 
with the operations defined by
\begin{itemize}
\item $x/\co \we y/\co := x \we y/\co$,
\item $x/\co \vee y/\co := x \vee y/\co$,
\item $x/\co \ra y/\co := x\ra y/\co$,
\item $0 := 0/\co$,
\item $1 := 1/\co$.
\end{itemize}

It is a natural question whether these constructions can be extended to subresiduated lattices.

\begin{definition} 
A subresiduated lattice (sr-lattice for short) is a pair $(A, D)$, where $A$ 
is a bounded distributive lattice and $D$ is a bounded sublattice of $A$ such 
that for each $a, b \in A$ there exists the maximum of the set $\{d\in D:a\we d\leq b\}$. 
This element is denoted by $a\rightarrow b$.
\end{definition}

Let $(A,D)$ be a subresiduated lattice. This pair can be regarded as an algebra 
$\langle A,\wedge,\vee,\ra,0,1 \rangle$ of type $(2,2,2,0,0)$ where 
$D = \{a \in A : 1 \rightarrow a = a\}=\{1\rightarrow a : a\in A\}$. 
Moreover, an algebra $\langle A,\wedge, \vee, \rightarrow, 0, 1\rangle$ is an sr-lattice if and only if 
$(A,\wedge, \vee, 0, 1)$ is a bounded distributive lattice and the 
following conditions are satisfied for every $a,b,c\in A$:
\begin{enumerate}[\normalfont 1)]
\item $(a\vee b)\ra c=(a\ra c)\we (b\ra c)$,
\item $c\ra (a\we b)=(c\ra a)\we (c\ra b)$,
\item $(a\ra b)\we (b\ra c)\leq a\ra c$,
\item $a\ra a=1$,
\item $a\we (a \ra b) \leq b$,
\item $a\ra b\leq c\ra (a\ra b)$.
\end{enumerate}

We write $\SRL$ to denote the variety whose members are sr-lattices.
In every sr-lattice the following cuasi-identity is satisfied:
if $a\leq b\ra c$ then $a\we b \leq c$.

The following example of sr-lattice will be used throughout the paper. 

\begin{example}\label{ex1}
Let $A$ be the Boolean algebra of four elements, where $a$ and $b$ are the atoms. This algebra can be 
seen as a bounded distributive lattice.
Define $D = \{0,1\}$. We have that $(A,D)$ is an sr-lattice. With an abuse of notation we write 
$A$ for this sr-lattice. Note that since $a\ra 0 = 0 \neq b$ then $A$ is not a Heyting algebra.
\end{example}

In this work, we attempt to find a more general definition than the one of a Nelson algebra in order to take the first steps toward exploiting its relation to sr-lattices.

\section{Subresiduated Nelson algebras}\label{s2}

In this section we define subresiduated Nelson algebras and we show 
that the class of subresiduated Nelson algebras, which is a variety, properly contains the variety 
of Nelson algebras. We also prove that every subresiduated Nelson algebra can be represented as a 
twist structure of a subresiduated lattice.

Let $A\in \SRL$. We define $\K(A)$ as in \eqref{ka}. 
Then $\langle \K(A),\we,\vee,\s,(0,1),(1,0) \rangle$ is a Kleene algebra \cite{C,K}.
On $\K(A)$ we also define the binary operation $\Ra$ as in Section \ref{s1}.
Note that this is a well defined map because if $(a,b)$ and $(c,d)$ are elements of $\K(A)$
then $(a\ra c)\we a\we d \leq c\we d = 0$, i.e., $(a\ra c)\we a\we d = 0$.
Thus, the structure $\langle \K(A),\we,\vee,\Ra,\s,(0,1),(1,0) \rangle$ is an algebra of type
$(2,2,2,1,0,0)$.

\begin{remark} \label{fc}
Let $A\in \SRL$ and $(a,b)$, $(c,d)$ in $\K(A)$. Then $(a,b)\Ra (c,d) = (1,0)$
if and only if $a\leq c$.
\end{remark}

\begin{definition}\label{generalised}
An algebra $\langle T, \we, \vee, \ra, \s, 0, 1\rangle$ of type $(2, 2, 2, 1, 0, 0)$ is said to be a
\textit{subresiduated Nelson algebra} if $\langle T, \wedge, \vee, \sim, 0, 1\rangle$ is a Kleene algebra and the
following conditions are satisfied for every $a,b,c\in T$:
\begin{enumerate}[\normalfont 1)]
\item $(x\vee y)\ra z = (x\ra z)\we (y\ra z)$,\label{supremo}
\item $z\ra (x\we y) = (z\ra x)\we (z\ra y)$,\label{infimo}
\item $((x\ra y)\we (y\ra z))\ra (x\ra z)=1$, \label{trans}
\item $x\ra x=1$, \label{imp1}
\item $x\we (x\ra y)\leq x\we (\s x\vee y)$, \label{inf}
\item $x\ra y\leq z\ra (x\ra y)$, \label{impl2}
\item $\s(x\ra y)\ra (x\we \s y)=1$, \label{au}
\item $(x\we \s y)\ra \s(x\ra y)=1$. \label{ultima}
\end{enumerate}
\end{definition}

We write $\SNA$  to denote the variety whose members are subresiduated Nelson algebras.

\begin{proposition} \label{Nelson}
The variety $\NA$ is a subvariety of $\SNA$.
\end{proposition}

\begin{proof}
Let $T\in \NA$ and $x,y,z \in T$. Condition 1) of Definition \ref{generalised} is condition (1.9) of \cite{Vig},
2) is Ne8), 4) is Ne4), 5) is a direct consequence of Ne6), 7) is (1.24) of \cite{Vig} and 8) is (1.23) of \cite{Vig}.

Now we will see 3), i.e, we will show that 
\begin{equation}\label{Ne0}
((x\ra y)\we (y\ra z))\ra (x\ra z)=1.
\end{equation}
It follows from (1.17) of \cite{Vig} that \eqref{Ne0} holds
if and only if $(y\ra z)\ra ((x\ra y)\ra (x\ra z)) = 1$.
But it follows from (1.11) of \cite{Vig} and Ne5) that
\[
(x\ra y)\ra (x\ra z) = x\ra (y\ra z) = (x\we y) \ra z.
\]
Hence, \eqref{Ne0} holds if and only if $(y\ra z)\ra ((x\we y)\ra z) = 1$.
Since $x\we y \leq y$ then it follows from (1.7) of \cite{Vig} that $y\ra z \leq (x\we y) \ra z$.
Thus, (1.3) of \cite{Vig} shows that $(y\ra z)\ra ((x\we y)\ra z) = 1$, so 3) is satisfied.

Finally we will show 6). Note that it follows from Ne5) that
\begin{equation}\label{Ne1}
z\ra (x\ra y) = (z\we x)\ra y.
\end{equation}
Besides, since $z\we x \leq x$ then it follows from (1.7) of \cite{Vig}
that
\begin{equation}\label{Ne2}
x\ra y \leq (z\we x)\ra y.
\end{equation}
Thus, from \eqref{Ne1} and \eqref{Ne2} we get $x\ra y \leq z\ra (x\ra y)$, which is 6).
\end{proof}

\begin{proposition} \label{Kalman}
If $A\in \SRL$ then $\K(A)\in \SNA$.
\end{proposition}

\begin{proof}
Let $A\in \SRL$. We will show that $\K(A)$ satisfies the conditions of Definition
\ref{generalised}. In order to prove it, let $x=(a,b)$, $y = (c,d)$ and $z = (e,f)$ be elements of $\K(A)$.

A direct computation shows that
\[
(x\vee y)\ra z = ((a\vee c)\ra e, (a\vee c)\we f),
\]
\[
(x\ra z)\we (y\ra z) = ((a\ra e)\we (c\ra e),(a\we f) \vee (c\we f)).
\]
Since $(a\vee c)\ra e = (a\ra e)\we (c\ra e)$ and $(a\vee c)\we f = (a\we f) \vee (c\we f)$
then $(x\vee y)\ra z = (x\ra z)\we (y\ra z)$, so 1) is satisfied.
The fact that $z\ra (x\we y) = (z\ra x) \we (z\ra y)$ can be proved following a similar reasoning, so 2) is also satisfied.
In order to show 3), note that it follows from Remark \ref{fc} that $((x\ra y)\we (y\ra z))\ra (x\ra z) = (1,0)$
if and only if $(a\ra c)\we (c\ra e)\leq a\ra e$, and the last inequality holds in sr-lattices. Hence, condition 3) holds.
Taking into account that $a\ra a = 1$ we get $x\ra x = (1,0)$, so 4) is verified too. In order to see 5), note that a straightforward computation shows that
\[
x\we (x\ra y) = (a\we (a\ra c), b\vee (a\we d)),
\]
\[
x\we (\sim x \vee y) = (a\we (b\vee c), b\vee (a\we d)).
\]
Since $a\we (a\ra c)\leq a\we c = (a\we b) \vee (a\we c) = a\we (b\vee c)$ then it follows from Remark \ref{fc} that
$((x\we (x\ra y))\ra (x\we (\sim x \vee y)) = (1,0)$. Hence, we have proved 5).
Now we will show 6). Note that
\[
x\ra y = (a\ra c, a\we d)
\]
\[
z\ra (x\ra y) = (e\ra (a\ra c),e\we (a\we d)).
\]
Since $a\ra c \leq e\ra (a\ra c)$ and $e\we a\we d\leq a\we d$ then
$x\ra y \leq z\ra (x\ra y)$. Hence, 6) is satisfied. In order to prove 7) and 8),
note that
\[
\sim (x\ra y) = (a\we d, a\ra c),
\]
\[
x\we \sim y = (a\we d,b\vee c).
\]
Since $a\we d = a\we d$ then it follows from Remark \ref{fc} that $\sim (x\ra y)\ra (x\we \sim y) = (1,0)$ and
$(x\we \sim y) \ra \sim (x\ra y) = (1,0)$. Thus, we have proved 8).

Therefore, $\K(A)\in \SNA$.
\end{proof}

It is important to note that the variety $\NA$ is a proper subvariety of $\SNA$. 
Indeed, let $A$ be the subresiduated lattice given in Example \ref{ex1}. 
In particular, $\K(A) \in \SNA$. Take $x = (1,0)$ and $y = (a,b)$, which are elements of $\K(A)$.
A direct computation shows that $x\we (x\ra y) = (0,b) \neq x\we (\sim x \vee y)$, so condition N8) is not
satisfied. Therefore, $\K(A) \notin \NA$.

\begin{proposition} \label{pvi}
Let $T\in \SNA$.
The following conditions are satisfied for every $x,y,z\in T$:
\begin{enumerate}
\item $1\ra x\leq x$,
\item if $x\leq y$ then $z\ra x\leq z\ra y$ and $y\ra z \leq x \ra z$,
\item if $x\leq y$ then $x\ra y = 1$,
\item $(x\we (x\ra y))\ra y = 1$,
\item if $x\ra y = 1$ then $x=x\we (\s x\vee y)$,
\item if $x\ra y = 1$ and $\s y\ra \s x=1$ then $x\leq y$,
\item if $x\ra y=1$ and $y\ra z=1$ then $x\ra z=1$,
\item if $x\ra y=1$ then $(x\we z)\ra (y\we z)=1$ and $(x\vee z)\ra (y\vee z)=1$,
\item if $x\ra y=1$ then $(y\ra z)\ra (x\ra z)=1$ and $(z\ra x)\ra (z\ra y)=1$.
\end{enumerate}
\end{proposition}

\begin{proof}
1. By \ref{inf}) of Definition \ref{generalised},
\[
(1\ra x)=1\we (1\ra x)\leq 1\we(\s 1 \vee x)=x.
\]

2. Suppose that $x\leq y$. Then $x=x\we y$, so it follows from \ref{infimo}) of Definition \ref{generalised} that
\[
z\ra x=z\ra (x\we y)=(z\ra x)\we (z\ra y)\leq z\ra y.
\]
Using \ref{supremo}) of Definition \ref{generalised}, the proof of the other inequality is analogous.

3. It follows from the previous item and \ref{imp1}) of Definition \ref{generalised}.

4. It follows from \ref{inf}) that $x\we (x\ra y) \leq x\we (\s x \vee y)$, so
\begin{equation}\label{eq1}
(x\we (\s x \vee y)) \ra y \leq ((x\we (x\ra y))\ra y.
\end{equation}
Also note that $(x\we \s x) \ra y = 1$.
Indeed, since $0\leq y$ then $(x\we \s x) \ra 0 \leq (x\we \s x)\ra y$.
But it follows from \ref{ultima}) that
\[
(x\we \s x) \ra 0 = (x\we \s x) \ra \s (x\ra x) = 1,
\]
so
\begin{equation}\label{eq2}
(x\we \s x) \ra y = 1.
\end{equation}
Thus, 
$(x\we (\s x \vee y)) \ra y = ((x\we \s x)\vee (x\we y)) \ra y$. It follows from \eqref{eq2} and the previous item that  $((x\we \s x)\vee (x\we y)) \ra y= ((x\we \s x) \ra y)\we ((x\we y)\ra y)=1$.

Hence, it follows from \eqref{eq1} that
$1\leq (x\we (x\ra y))\ra y$, i.e., $(x\we (x\ra y))\ra y = 1$.

5. It follows from \ref{inf}) of Definition \ref{generalised}.

6. Suppose that $x\ra y = 1$ and $\s y\ra \s x = 1$. It follows from 5.
of this proposition that
\[
x = x\we (\s x \vee y) = (x\we \s x) \vee (x\we y),
\]
\[
\s y = \s y \we (y \vee \s x) =(\s y\we y)\vee (\s y \we \s x),
\]
so
\[
y = (y\vee \s y) \we (x\vee y).
\]
In particular,
\[
x =(x\we \s x)\vee (x\we y) \leq (y\vee \s y) \vee (x\we y) = y\vee \s y,
\]
so
\[
x = x\we (x\vee y) \leq (y\vee \s y) \we (x\vee y) = y.
\]
Therefore, $x\leq y$.

7. Suppose that $x\ra y=1$ and $y\ra z=1$. By \ref{trans}),
\[
1\ra (x\ra z)=((x\ra y)\we (y\ra z))\ra (x\ra z)=1.
\]
Thus, it follows from 1. of this proposition that $1=1\ra (x\ra z)\leq x\ra z$.
Then $x\ra z = 1$.

8. Suppose that $x\ra y=1$. Thus, it follows from \ref{infimo}) that
\[
(x\we z)\ra (y\we z)=((x\we z)\ra y)\we ((x\we z)\ra  z))=(x\we z)\ra y.
\]
From $x\we z\leq x$ we get $1=x\ra y\leq (x\we z)\ra y$ and therefore $(x\we z)\ra (y\we z)=1$.
The other implication is analogous using \ref{supremo}).

9. Suppose that $x\ra y=1$. Then by \ref{trans}),
\[
1=((x\ra y)\we (y\ra z))\ra (x\ra z)=(y\ra z)\ra (x\ra z).
\]
Also by \ref{trans}),
\[
1=((z\rightarrow x)\we (x\ra y))\ra (z\ra y)=(z\ra x)\ra (z\ra y).
\]
\end{proof}

Let $A\in \SNA$. We define the binary relation $\co$ as in \eqref{rel} of Section \ref{s1}.

\begin{lemma} \label{l1}
Let $T\in \SNA$. Then $\co$ is an equivalence relation compatible with $\we$, $\vee$ and $\ra$.
\end{lemma}

\begin{proof}
The reflexivity and symmetry of $\co$ are immediate. The transitivity of the relation
follows from 7. of Proposition \ref{pvi}. Thus, $\co$ is an equivalence relation.
In order to show that $\co$ is compatible with $\we$, $\vee$ and $\ra$, let $x,y,z\in T$ such that $x\co y$.
It follows from 8. of Proposition \ref{pvi} that $(x\we z,y\we z)\in \co$ and $(x\vee z,y\vee z) \in \co$, so
$\co$ is compatible with $\we$ and $\vee$. Finally, it follows from 9. of the previous proposition
that $(x\ra z,y\ra z)\in \theta$ and $(z\ra x,z\ra y) \in \co$, which implies that $\co$ is compatible with respect to $\ra$.
\end{proof}

Let $\langle T, \we, \vee, \ra, \s, 0, 1\rangle$ be a subresiduated Nelson algebra.
Then it follows from Lemma \ref{l1} that we can define on $T/\theta$ the operations 
$\we$, $\vee$, $\ra$, $0$ and $1$ as in Section \ref{s1}.
In particular, $\langle T/\co, \we, \vee, 0, 1\rangle$ is a bounded distributive lattice.
We denote by $\preceq$ to the order relation of this lattice.

\begin{lemma}\label{l2}
Let $T\in \SNA$ and $x,y\in A$. Then $x/\co \preceq y/\co$ if and only if $x\ra y = 1$.
Moreover, $x\ra y = 1$ if and only if $x\ra y/\co = 1/\co$.
\end{lemma}

\begin{proof}
Suppose that $x/\co \preceq y/\co$, i.e., $x/\co = x\we y/\co$, so $x\ra (x\we y) = 1$.
But $x\ra (x\we y) = (x\ra y)\we (x\ra x) = (x\ra y)\we 1 = x\ra y$, so
$x\ra y = 1$.
Conversely, suppose that $x\ra y = 1$. Thus, it follows from 8. of Proposition \ref{pvi}
that $(x\we x) \ra (x\we y) = 1$, i.e., $x\ra (x\we y) = 1$.
Since we also have that $(x\we y)\ra x = 1$, we conclude that $x/\co = x\we y/\co$, i.e.,
$x/\co \preceq y/\co$.

Finally suppose that $x\ra y/\co = 1/\co$, so $1\ra (x\ra y) = 1$. Therefore,
it follows from 1. of Proposition \ref{pvi} that $x\ra y = 1$.
\end{proof}

\begin{proposition}
If $T\in \SNA$ then $T/\co \in \SRL$.
\end{proposition}

\begin{proof}
Let $T\in \SNA$ and $x,y,z\in T$. We only need to show that the inequalities
$x\we (x\ra y)/\co \preceq y/\co$ and $(x\ra y)\we (y\ra z)/\co \preceq x\ra z/\co$
are satisfied. It follows from Lemma \ref{l1} that $x\we (x\ra y)/\co \preceq y/\co$
if and only if $(x\we (x\ra y))\ra y =  1$. Since by 4. of Proposition
\ref{pvi} we have that the previous equality holds, so $x\we (x\ra y)/\co \preceq y/\co$.
Finally, also note that $(x\ra y)\we (y\ra z)/\co \preceq x\ra z/\co$ if and only if
$((x\ra y)\we (y\ra z))\ra (x\ra z) = 1$. But the equality $((x\ra y)\we (y\ra z))\ra (x\ra z) = 1$
is satisfied, so $(x\ra y)\we (y\ra z)/\co \preceq x\ra z/\co$.
\end{proof}

\begin{theorem}\label{rept}
Let $T\in \SNA$. Then the map $\rho_T:T\ra \K(T/\co)$ given by $\rho_{T}(x) = (x/\co,\s x/\co)$ is a monomorphism.
\end{theorem}

\begin{proof}
We write $\rho$ in place of $\rho_T$.
First we will show that $\rho$ is a well defined map. Let $x\in T$.
Note that $x/\co \we \s x/\co = 0/\co$ if and only if $(x\we \s x) \ra 0 = 1$.
It follows from \ref{ultima}) of Definition \ref{generalised}
that $(x\we \s x) \ra 0 = (x\we \s x) \ra \s(x\ra x) = 1$.
Thus, $\rho$ is well defined. The fact that $\rho$ is a bounded lattice homomorphism is immediate.
It is also immediate that $\rho$ preserves $\s$.
Moreover, \ref{au}) and \ref{ultima}) of Definition \ref{generalised} show that $h$ preserves the implication.
Hence, $\rho$ is a homomorphism. Finally, a direct computation based in item 6. of Proposition \ref{pvi}
proves that $\rho$ is an injective map.
Therefore, $\rho$ is a monomorphism.
\end{proof}

\section{Congruence relations on subresiduated Nelson algebras} \label{s4}

We start this section by giving some elemental definitions.

\begin{definition}
Let $T\in \SNA$ and $F\subseteq T$. 
\begin{enumerate}[\normalfont 1)]
\item We say that $F$ is a filter of $T$ if $1\in F$, $F$ is an upset (i.e.,
for every $x, y \in T$, if $x\leq y$ and $x\in F$ then $y\in F$) and $x\we y \in F$, for all $x, y\in F$.
If in addition $1\ra x\in F$ for every $x\in F$, we say that $F$ is an open filter of $T$.
\item We say that $F$ is an implicative filter of $T$
if $1\in F$ and for every $x,y \in F$, if $x\in F$ and $x\ra y \in F$ then $y\in F$.
If in addition $1\ra x\in F$ for every $x\in F$, we say that $F$ is an open implicative filter of $T$.
\end{enumerate}
\end{definition}

In this section we prove that for every $T\in \SNA$ there exists an order isomorphism
between the lattice of congruences of $A$ and the lattice of open implicative filters of $T$.
We use it in order to give a characterization of the principal congruences of $T$.
In particular, the mentioned characterization proves that the variety $\SNA$ has EDPC. 
We also give a description of the simple and subdirectly irreducible algebras of $\SNA$, 
and we prove that the variety $\SNA$ has CEP.

We start by giving some elemental properties of sr-lattices, which will then be used to transfer them into the framework of subresiduated Nelson algebras.

Let $A\in \SRL$ and $a\in A$. We define $\sq a:= 1\ra a$.
In the same way, given $T\in \SNA$ and $x\in T$, we define 
$\sq x:= 1\ra x$.

\begin{lemma}\label{lc1}
Let $A\in \SRL$ and $a,b,c\in A$. Then the following conditions are satisfied:
\begin{enumerate}[\normalfont 1)]
\item $a\ra (b\ra c) \leq (a\ra b)\ra (a\ra c)$,
\item $\square a\ra (\square b\ra c) = \square b\ra (\square a\ra c)$,
\item $\square b \leq a \ra (a\we b)$.
\end{enumerate}
\end{lemma}

\begin{proof}
Conditions 1) and 2) follow from results of \cite{CFMSM}.
Condition 3) follows from a direct computation. 
\end{proof}

\begin{lemma}\label{lc1b}
Let $T\in \SNA$ and $v,w,x,y,z\in T$. Then for every $x,y,z\in T$ the following conditions are satisfied:
\begin{enumerate}[\normalfont 1)]
\item $x\ra (y\ra z) \leq (x\ra y)\ra (x\ra z)$,
\item $(x \ra v)\ra ((y\ra w)\ra z) = (y\ra w) \ra ((x\ra v)\ra z)$,
\item $\sq(x\ra y) = x\ra y$,
\item $\sq y \leq x \ra (x\we y)$,
\item $\sq x \leq \s x\ra 0$.
\end{enumerate}
\end{lemma}

\begin{proof}
By Theorem \ref{rept} we can assume that $T$ is a subalgebra of $\K(A)$ for some $A\in \SRL$.
Conditions 1) and 2) follows from a direct computation based in Lemma \ref{lc1}.
Condition 3) and 4) follows from a direct computation.
In order to show 4), let $x=(a,b) \in T$. 
We will see that $\sq x \leq \s x\ra 0$, i.e., $\sq a \leq b\ra 0$. Since $\sq a\we b \leq a\we b = 0$ then
$\sq a \leq b\ra 0$, and hence our result is proved.  
\end{proof}

Let $T\in \SNA$. We write $\Con(T)$ to indicate the set of congruences of $T$. Given $\theta \in \Con(T)$ and $x\in T$,
we write $x/\theta$ for the equivalence class of $x$ associated to the congruence $\theta$.

\begin{lemma} \label{lc2}
Let $T\in \SNA$, $\theta \in \Con(T)$ and $x,y \in T$. Then $(x,y)\in \theta$ if and only if
$x\ra y$, $\s y \ra \s x$, $y\ra x$, $\s x \ra \s y \in 1/\theta$.
\end{lemma}

\begin{proof}
Let $\theta \in \Con(T)$ and $x,y \in T$. It is immediate that if $(x,y)\in \theta$ then
$x\ra y$, $\s y \ra \s x$, $y\ra x$, $\s x \ra \s y \in 1/\theta$.
Conversely, assume that $x\ra y$, $\s y \ra \s x$, $y\ra x$, $\s x \ra \s y \in 1/\theta$.
Following a similar reasoning than the one employed in item 6. of Proposition \ref{pvi} it can be proved that $(x, (x\we\s x)\vee (x\we y))\in \theta$
and $(y,(y\vee \s y) \we (x\vee y))\in \theta$. Taking into account the inequality $x\we \s x \leq y \vee \s y$ and
the distributivity of the underlying lattice of $A$ we get $(x\we y,x)\in \theta$. Similarly it can be showed that $(y\we x,y)\in \theta$.
Therefore, $(x,y)\in\theta$.
\end{proof}

\begin{lemma} \label{lc3}
Let $T\in \SNA$ and $\theta,\psi \in \Con(T)$. Then $\theta\subseteq \psi$ if and only if
$1/\theta \subseteq 1/\psi$. In particular, $\theta = \psi$ if and only if $1/\theta = 1/\psi$.
\end{lemma}

\begin{proof}
Let $\theta,\psi \in \Con(T)$. It is immediate that if $\theta\subseteq \psi$ then
$1/\theta \subseteq 1/\psi$. The converse follows from a direct computation based in Lemma \ref{lc2}.
\end{proof}

\begin{lemma} \label{lc4}
Let $T\in \SNA$ and $F\subseteq T$. If $F$ is an open implicative filter then $F$ is an open filter.
\end{lemma}

\begin{proof}
Let $F$ be an open implicative filter. In order to show that $F$ is an upset, let $x,y\in T$ be such that
$x\in F$ and $x\leq y$. Then $x\ra y = 1 \in F$, so $y\in F$. Hence, $F$ is an upset.
Finally, let $x,y\in F$. We will see that $x\we y \in F$. Since $F$ is open then $1\ra y \in F$.
From \ref{lc1b}, we get that $1\ra y\leq x \ra (x\we y)$, so $x\ra (x\we y) \in F$. Using that $x\in F$ we obtain that
$x\we y \in F$. Therefore, $F$ is an open filter.
\end{proof}

Let $T\in \SNA$ and $F$ an implicative filter of $T$. For every $x,y \in T$ we define
$s(x,y) = (x\ra y)\we (y\ra x) \we (\s x \ra \s y) \we (\s y \ra \s x)$. We also define the set
\[
\Theta(F) = \{(x,y)\in T\times T:s(x,y)\in F\}.
\]
Note that $s(x,y)\in F$ if and only if $x\ra y, y\ra x, \s x \ra \s y, \s y \ra \s x \in F$.

\begin{lemma}\label{lc5}
Let $T\in \SNA$ and $F$ be an open implicative filter. Then $\Theta(F) \in \Con(T)$.
\end{lemma}

\begin{proof}
It is immediate that $\Theta(F)$ is reflexive and symmetric. In order to show that it is transitive,
let $(x,y), (y,z)\in \Theta(F)$. Since it follows from Lemma \ref{lc4} that $F$ is a filter and $x\ra y, y\ra z \in F$ then
$(x\ra y)\we (y\ra z) \in F$. But $((x\ra y)\we (y\ra z))\ra (x\ra z) = 1\in F$, so $x\ra z \in F$. In a similar way it can be proved that 
$z\ra x, \s z \ra \s x, \s x \ra \s z \in F$. Thus, $(x,z) \in \Theta(F)$. Hence, $\Theta(F)$ is an equivalence relation.

Now we will show that $\Theta(F)$ is a congruence. 
Let $x,y,z\in T$ be such that $(x,y) \in \Theta(F)$. First we will show that $(x\vee z, y\vee z) \in \Theta(F)$.
Note that $(x\vee z) \ra (y\vee z) = (x\ra (y\vee z))\we (z\ra (y\vee z)) = x\ra (y\vee z) \geq x\ra y$.
Since $x\ra y \in F$ and $F$ is an upset we get $(x\vee z) \ra (y\vee z)\in F$. In a similar way it can be proved 
that $(y\vee z) \ra (x\vee z)\in F$.
Besides, note that since 
\[
\s (x\vee z) \ra \s (y\vee z) = (\s x \we \s z) \ra (\s y \we \s z) = (\s x \we \s z) \ra \s y,
\]
and $(\s x \we \s z) \ra \s y \geq \s x \ra \s y$ then $\s x \ra \s y \leq \s (x\vee z) \ra \s (y\vee z)$.
But $\s x \ra \s y \in F$, so $\s (x\vee z) \ra \s (y\vee z)\in F$. Analogously it can be showed that
$\s (y\vee z) \ra \s (x\vee z)\in F$. Thus,  $(x\vee z, y\vee z) \in \Theta(F)$. A similar argument
proves that  $(x\we z, y\we z) \in \Theta(F)$. It is also immediate that $(\s x, \s y)\in \Theta(F)$.
We have proved that $\Theta(F)$ preserves the operations $\vee$, $\we$ and $\s$.

Following this, we will see that $\Theta(F)$ preserves $\ra$. In order to show it,
let $x,y,z\in T$ such that $(x,y)\in \Theta(F)$. First we will prove that $(z\ra x,z\ra y)\in \Theta(F)$. 
Since $x\ra y \leq z\ra (x\ra y)$ and $x\ra y\in F$ then $z\ra (x\ra y)\in F$. It follows from Lemma \ref{lc1b}
that $z\ra (x\ra y) \leq (z\ra x) \ra (z\ra y)$, so $(z\ra x) \ra (z\ra y) \in F$. In a similar way we can show
that $(z\ra y) \ra (z\ra x) \in F$. Now we will prove that $\s (x\ra z) \ra \s (y\ra z) \in F$. 
First note that since $\s y \ra \s x \leq (\s y \we z) \ra (\s x \we z)$
and $\s y \ra \s x \in F$ then $(\s y \we z) \ra (\s x \we z) \in F$.
We also have that $(z\we \s x) \ra \s (z\ra x) = 1$. Taking into account that
\[
(((\s y \we z)\ra (\s x \we z)) \we ((\s x\we z) \ra \s (z\ra x))) \ra ((\s y \we z) \ra \s (z\ra x)) = 1
\]
we get 
\[
((\s y \we z)\ra (\s x \we z)) \ra ((\s y \we z) \ra \s (z\ra x)) = 1.
\]
But $(\s y \we z) \ra (\s x \we z) \in F$ and $1\in F$, so 
\begin{equation}\label{eq-i1}
(\s y \we z) \ra \s (z\ra x) \in F.
\end{equation}
Besides, since 
\begin{multline*}
    ((\s (z\ra y) \ra (\s y\we z)) \we ((\s y \we z)\ra \s (z\ra x)))\ra\\ (\s (z\ra y) \ra \s (z\ra x)) = 1
\end{multline*}

and $\s (z\ra y) \ra (\s y\we z) = 1$,
so
\[
((\s y \we z)\ra \s (z\ra x))\ra (\s (z\ra y) \ra \s (z\ra x)) = 1 \in F.
\]
Then it follows from (\ref{eq-i1}) that $\s (z\ra y) \ra \s (z\ra x) \in F$.
In a similar way we can see that $\s (z\ra x) \ra \s (z\ra y) \in F$.
Hence, $(z\ra x,z\ra y)\in \Theta(F)$. Finally we will show that
$(x\ra z,y\ra z)\in \Theta(F)$. It follows from Lemma \ref{lc1b} that
\[
(y\ra z) \ra ((x\ra z) \ra (y\ra z)) = (x\ra z) \ra ((y\ra z) \ra (y\ra z)) = 1 \in F. 
\]
Thus, since $y\ra z\in F$ then $(x\ra z) \ra (y\ra z)\in F$. Analogously,
$(y\ra z) \ra (x\ra z) \in F$. Now we will see that $\s (x\ra z) \ra \s (y\ra z) \in F$.
Note that
\begin{multline*}
((\s (x\ra z)\ra (\s z\we x)) \we ((\s z \we x)\ra \s (y\ra z)))\ra\\ (\s (x\ra z) \ra \s (y\ra z)) = 1.    
\end{multline*}

But $\s (x\ra z)\ra (\s z\we x) = 1$, so
\[
((\s z \we x)\ra \s (y\ra z))\ra (\s (x\ra z) \ra \s (y\ra z)) = 1.
\]
Since $1\in F$ and $F$ is a filter, in order to show that $\s (x\ra z) \ra \s (y\ra z) \in F$ it is enough
to see that $(\s z \we x)\ra \s (y\ra z)\in F$. Since $x\ra y \leq (x\we \s z) \ra (y\we \s z)$ and $x\ra y\in F$
then 
\begin{equation}\label{eq-c2}
(x\we \s z) \ra (y\we \s z)\in F. 
\end{equation}
Besides, since 
$(((\s z\we x)\ra (y\we \s z))\we ((y\we \s z)\ra \s (y\ra z))) \ra ((\s z \we x)\ra \s (y\ra z)) = 1$
and $(y\we \s z)\ra \s (y\ra z) = 1$ then
\[
((\s z\we x)\ra (y\we \s z))) \ra ((\s z \we x)\ra \s (y\ra z)) = 1.
\]
Since $1\in F$ then it follows from (\ref{eq-c2}) that $(\s z \we x)\ra \s (y\ra z)\in F$, which was our aim.
Then $\s (x\ra z) \ra \s (y\ra z) \in F$. Analogously, $\s (y\ra z) \ra \s (x\ra z) \in F$.
 Hence, $\Theta(F)$ preserves $\ra$.
\end{proof}

For $T\in \SNA$ we write $\oIF(T)$ to denote the set of open implicative filters of $T$.

\begin{theorem}\label{thm-con}
Let $T\in \SNA$. The assignments $\theta \mapsto 1/\theta$ and $F\mapsto \Theta(F)$ establish an
order isomorphism between $\Con(T)$ and $\oIF(T)$. 
\end{theorem}

\begin{proof}
Let $H:\Con(T)\ra \oIF(T)$ be the function given by $H(\theta) = 1/\theta$. In order to show that $H$ is a well 
defined map, let $\theta \in \Con(T)$. In particular, $1\in 1/\theta$. Let $x, x\ra y \in 1/\theta$,
so $x\we (x\ra y) \in 1/\theta$. Since $(x\we (x\ra y))\ra y = 1$ then $1\ra y \in 1/\theta$. But $1\ra y \leq y$,
and $(y\we (1\ra y),y)\in \theta$, so $(1\ra y,y)\in \theta$. Thus, $y\in 1/\theta$. Hence, $1/\theta$ is an implicative filter.
The fact that $1/\theta$ is open is immediate, so $1/\theta \in \oIF(A)$. Hence, $H$ is a well defined map.
The injectivity of $H$ follows from Lemma \ref{lc2}. In order to show that $H$ is suryective, let $F\in \oIF(T)$.
Then it follows from Lemma \ref{lc5} that $\Theta(F)\in \Con(A)$. We will show that $H(\Theta(F)) = F$, i.e.,
$1/\Theta(F) = F$. In order to prove it, let $x\in 1/\Theta(F)$, i.e., $(x,1)\in \Theta(F)$. In particular, $1\ra x\in F$.
But $1\ra x\leq x$, so $x\in F$. Conversely, let $x\in F$. In particular, $x \ra 1 = 1\in F$. Besides, since $F$
is open then $1\ra x \in F$. We also have that $\s 1 \ra \s x = 1 \in F$. Finally, it follows from 
Lemma \ref{lc1b} that $1\ra x \leq  \s x \ra 0 = \s x \ra \s 1$, so $\s x \ra \s 1\in F$. Then $x\in 1/\Theta(F)$.
Thus, $1/\Theta(F) = F$. We have proved that $H$ is a suryective map. Hence, $H$ is a bijective function.
Therefore, it follows from Lemma \ref{lc2} that $H$ is an order isomorphism.
\end{proof}

Let $T\in \SNA$ and $X\subseteq T$. We write $\langle X \rangle$ in order to indicate the open implicative filter generated by $X$,
i.e., the least open implicative filter (with respect to the inclusion) which contains the set $X$.
In other words, $\langle X \rangle$ is the intersection of all the open implicative filters that contain $X$.

\begin{lemma}\label{ac1}
Let $T\in \SNA$ and $X$ a non empty subset of $T$. Then
\[
\langle X \rangle = \{x\in T:\sq(x_1\we\cdots \we x_n)\ra x = 1\;\text{for some}\;x_1,\ldots,x_n\in X\}.
\]
\end{lemma}

\begin{proof}
Let $S = \{x\in T:\sq(x_1\we\cdots \we x_n)\ra x = 1\;\text{for some}\;x_1,\ldots,x_n\in X\}$. We will show that
$S$ is an open implicative filter. It is immediate that $1\in S$. Let now $x, y\in T$ be such that
$x, x\ra y \in S$. Then there exist $x_1,\ldots,x_n,y_1,\ldots,y_m \in X$ such that
$\sq(x_1\we \cdots \we x_n) \ra x = 1$ and $\sq(y_1\we \cdots \we y_m) \ra (x\ra y) = 1$.
Let $z = x_1\we\cdots \we x_n\we y_1\we \cdots y_m$.
Since 
$\sq(z)\leq \sq(x_1\we\cdots \we x_n), \sq(y_1\we \cdots y_m)$  
then $\sq(z) \ra x = 1$ and $\sq(z) \ra (x\ra y) = 1$. Thus, it follows from Lemma \ref{lc1b}
that
\[
1 = \sq(z) \ra (x\ra y) \leq (\sq(z) \ra x)\ra (\sq(z)\ra y) = 1 \ra (\sq(z) \ra y) 
\]
so $\sq(z) \ra y =1 \ra (\sq(z) \ra y)= 1$. Hence, $y\in S$. We have proved that $S$ is an implicative filter.
In order to see that $S$ is open, let $x\in S$ so $\sq(x_1\we\cdots \we x_n)\ra x = 1$ for some $x_1,\ldots,x_n\in X$.
Then 
\[
1 = \sq(1) = \sq(\sq(x_1\we\cdots \we x_n)\ra x) \leq \sq(x_1\we \cdots \we x_n) \ra \sq(x), 
\]
so $\sq(x_1\we \cdots \we x_n) \ra \sq(x) = 1$. Hence, $S$ is an open implicative filter.
Finally, let $F$ be an open implicative filter such that $X\subseteq F$. A direct computation based in the fact that every open implicative 
filter is an open filter shows that $S\subseteq F$. Therefore, $\langle X \rangle = S$, which was our aim. 
\end{proof}

A straightforward computation based in Lemma \ref{ac1} proves the following result.  

\begin{corollary} \label{ac2}
Let $T\in \SNA$, $F\in \oIF(T)$ and $x\in T$. Then 
\[
\langle F \cup \{x\} \rangle = \{y\in T:(f\we \square(x))\ra y = 1\;\text{for some}\;f\in F\}.
\]
\end{corollary}

Let $T\in \SNA$ and $x\in X$. We write $\langle x \rangle$ instead of $\langle \{x\} \rangle$. 

\begin{corollary}
Let $T\in \SNA$ and $x\in T$. Then $\langle x \rangle = \{y\in T: \square(x) \ra y = 1\}$.
\end{corollary}

Let $T \in \SNA$ and $x,y\in T$. For every $\theta \in \Con(T)$, it follows from Theorem \ref{thm-con} that $(x,y) \in \theta$
if and only if $(s(x,y),1) \in \theta$. Thus, $\SNA$ is a term variety where $(s(x,y),1)$ is a pair 
associated to $\SNA$.
\vspace{1pt}

Let $T\in \SNA$ and $x,y\in T$. We write $\theta(x,y)$ for the congruence generated by the pair $(x,y)$, i.e.,
the least congruence which contains the pair $(x,y)$ \cite{SM}.

\begin{proposition}
Let $T\in \SNA$ and $x,y,z,w\in T$. Then $(z,w)\in \theta(x,y)$ if and only if $s(x,y)\ra s(z,w) = 1$.
\end{proposition}

\begin{proof}
It follows from Theorem \ref{rept} and \cite[Theorem 2.4]{SM} that 
$(z,w)\in \theta(x,y)$ if and only if $s(z,w) \in \langle s(x,y)\rangle$. That is,
$(z,w)\in \theta(x,y)$ if and only if $\square(s(x,y)) \ra s(z,w) = 1$. But $\square(s(x,y)) = s(x,y)$,
so $(z,w)\in \theta(x,y)$ if and only if $s(x,y) \ra s(z,w) = 1$.
\end{proof}

\begin{corollary}
The variety $\SNA$ has $\mathrm{EDPC}$.
\end{corollary}

Other application of Theorem \ref{thm-con} are the following two propositions,
where we give a description of the simple and subdirectily irreducible algebras of
the variety $\SNA$ respectively.

\begin{proposition}
Let $T\in \SNA$. The following conditions are equivalent:
\begin{enumerate}[\normalfont 1)]
\item $T$ is simple.
\item For every $x\in T$, if $x\neq 1$ then $\square(x) \ra 0 = 1$. 
\end{enumerate}
\end{proposition}

\begin{proof}
Suppose that $T$ is simple, so $\oIF(T) = \{\{1\},T\}$. Let $x\in T$ such that $x\neq 1$.
Then $\langle x \rangle \neq \{1\}$, so $\langle x \rangle =T$. Since $0\in T$ then $\square(x)\ra 0 = 1$.
Conversely, let $F\in \oIF(T)$ be such that $F\neq \{1\}$, so there exists $x\in T$ such that
$x\neq 1$. It follows from hypothesis that $\square(x)\ra 0 = 1 \in F$. Since $x\in F$ then $\square(x)\in F$,
so $0\in F$, i.e., $F = T$. Therefore, $T$ is simple.
\end{proof}

\begin{proposition}
Let $T\in \SNA$ and suppose that $T$ is not trivial. The following conditions are equivalent:
\begin{enumerate}[\normalfont 1)]
\item $T$ is subdirectly irreducible.
\item There exists $x\in T-\{1\}$ such that for every $y\in T-\{1\}$, $\square(y) \ra x = 1$. 
\end{enumerate}
\end{proposition}

\begin{proof}
Let $T$ be a non trivial subresiduated Nelson algebra. Suppose that $T$ is subdirectly irreducible.
Thus, there exists $F\in \oIF(T)$ such that $F\neq \{1\}$ and for every open implicative filter $G\neq \{1\}$,
$F\subseteq G$. Since $F\neq \{1\}$ there exists $x\in F$ such that $x\neq 1$. Let $y\in T$ be
such that $y\neq 1$, i.e., $\langle y \rangle \neq \{1\}$. Thus, $x\in F \subseteq \langle y \rangle$,
i.e., $\square(y) \ra x = 1$. Conversely, suppose that 2) is satisfied. Let $x\neq 1$ be an element 
which satisfies 2). Let $F = \langle x \rangle$, so $F\neq \{1\}$. Let $G\in \oIF(T)$ be such that $G\neq \{1\}$.
Note that $F\subseteq G$ if and only if $x\in G$. In order to see that $x\in G$, note that since $G\neq \{1\}$
then there exists $y\in G$ such that $y\neq 1$. Thus, it follows from hypothesis that $\square(y)\ra x = 1 \in G$.
But $y\in G$, so $\square(y) \in G$. Hence, $x\in G$. Therefore, 
$T$ is subdirectly irreducible.
\end{proof}

Finally, we use Theorem \ref{thm-con} in order to show that 
$\SNA$ has the congruence extension property.

\begin{proposition} \label{CEP}
The variety $\SNA$ has the congruence extension property.
\end{proposition}

\begin{proof}
Let $T, U\in \SNA$ be such that $U$ is a subalgebra of $T$ and $\theta \in \Con(U)$.               
We will show that there exists $\hat{\theta} \in \Con(T)$ 
such that $\theta = \hat{\theta}\cap U^2$. We define $\hat{\theta}$ as the congruence of 
$T$ generated by $\theta$.
Thus, it follows from Lemma \ref{lc3} and Theorem \ref{thm-con} that $1/\hat{\theta}$ is the 
open implicative filter of $T$ generated by the set $1/\theta$. 
In order to see that $\theta = \hat{\theta}\cap U^2$,
let  $(x,y) \in \hat{\theta}\cap U^2$, so  $x,y \in U$ and $s(x,y) \in 1/\hat{\theta}$. 
In particular, $s(x,y)\in U$ and it follows from Lemma \ref{ac1} that there exist $x_1,\ldots,x_n \in 1/\theta$
such that $\sq(x_1\we\cdots \we x_n)\ra s(x,y) = 1$. Since $\sq(x_1\we\cdots \we x_n)$ and $1$ are elements of $1/\theta$
then $s(x,y)\in 1/\theta$, i.e., $(x,y) \in \theta$.
Therefore,  $\theta = \hat{\theta}\cap U^2$.
\end{proof}

\section{The variety generated by the class of totally ordered subresiduated Nelson algebras} \label{s5}

The aim of this final section is to give an equational base for the variety generated by the totally 
ordered subresiduated Nelson algebras.

Let $T\in \SNA$. A proper implicative filter $P$ of $T$ is called prime if for every $x,y\in T$
we have that if $x\vee y \in P$ then $x\in P$ or $y\in P$.  
We write $\oX(T)$ to denote the set of open prime implicative filters of $T$.

\begin{lemma} \label{lem-ch}
Let $T\in \SNA$ such that 
\[
T \vDash ((x\ra y)\we(\s y \ra \s x)) \vee ((y\ra x)\we(\s x \ra \s y)) = 1.
\] 
Let $P\in \oX(T)$. Then $T/P$ is a chain.
\end{lemma}

\begin{proof}
Let $P\in \oX(T)$ and $x,y\in T$. Taking into account that 
\[
((x\ra y)\we(\s y \ra \s x)) \vee ((y\ra x)\we(\s x \ra \s y)) = 1 \in P
\]
we get $(x\ra y)\we(\s y \ra \s x)\in P$ or $(y\ra x)\we(\s x \ra \s y) \in P$. Thus, $(x,x\we y)\in \Theta(P)$
or $(y,y\we x)\in \Theta(P)$, i.e., $x/P\leq y/P$ or $y/P\leq x/P$.
Hence, $T/P$ is a chain.
\end{proof}

\begin{lemma}
Let $T\in \SNA$ such that $T \vDash \square(x\vee y) = \square(x) \vee \square(y)$. Let $F\in \oIF(A)$ and $I$ an
ideal such that $F\cap I = \emptyset$. Then there exists a filter $P\in \oX(T)$ such that $F\subseteq P$ and $P\cap I = \emptyset$.
\end{lemma}

\begin{proof}
Let $\Sigma = \{G\in \oIF(T):F\subseteq G\;\text{and}\;G\cap I = \emptyset\}$. Since $F\in \Sigma$ then $\Sigma \neq \emptyset$.
A direct computation shows that $\Sigma$ is under the hypothesis of Zorn's lemma, so there exists a maximal element $P$ in $\Sigma$. 
In particular, $F\subseteq P$ and $P\cap I = \emptyset$. Moreover, it is immediate that $P$ is a proper open implicative filter.
Suppose that $P$ is not prime, so there exist $x,y\in T$ such that $x\vee y \in P$ and $x, y\notin P$.
Let $P_x =\langle P\cup \{x\}\rangle$ and $P_y =\langle P\cup \{y\}\rangle$. The maximality of $P$ in $\Sigma$ implies that
$P_x \cap I\neq \emptyset$ and $P_y \cap I\neq \emptyset$, so there exist $z,w\in I$ and $p_1,p_2\in P$ such that
$(p_1 \we \sq(x))\ra z = 1$ and $(p_2 \we \sq(y))\ra w = 1$. Let $p = p_1\we p_2 \in P$.
Thus, $(p\we \square(x))\ra (z\vee w) = 1$ and $(p\we \square(y))\ra (z\vee w) = 1$.
Thus,
\[ 
((p\we \sq(x))\ra (z\vee w))\we ((p\we \sq(y))\ra (z\vee w)) = 1.
\]
But 
\[
((p\we \sq(x))\ra (z\vee w))\we ((p\we \sq(y))\ra (z\vee w)) = ((p\we \square(x)) \vee (p\we \square(y)))\ra (y\vee z)
\]
and
\[
((p\we \square(x)) \vee (p\we \square(y)))\ra (y\vee z) = (p\we \square(x\vee y)) \ra (y\vee z).
\]
Then
\[
(p\we \square(x\vee y)) \ra (y\vee z) = 1 \in P
\]
and $p\we \square(x\vee y) \in P$, so $y\vee z\in P$. Besides, since $y,z\in I$ then $y\vee z\in I$,
so $P\cap I \neq \emptyset$, which is a contradiction. Therefore, $P\in \oX(T)$.
\end{proof}

Let $T \in \SNA$ be such that $T \vDash \square(x\vee y) = \square(x) \vee \square(y)$. Suppose that $T$ is not trivial.
Then $\oX(T)\neq \emptyset$. Indeed, let $x\in T$ such that $x\neq 1$. Then $(x] \cap \{1\} = \emptyset$, where 
$(x]:=\{y\in T:y\leq x\}$. Then there exists a filter $P\in \oX(T)$. Therefore, $\oX(T) \neq \emptyset$.

\begin{corollary} \label{cor-pf}
Let $T\in \SNA$ be such that $T \vDash \square(x\vee y) = \square(x) \vee \square(y)$. Suppose that $T$ is not trivial.
Then the intersection of all open prime implicative filters of $T$ is equal to $\{1\}$.
\end{corollary}

Let $T\in \SNA$ and $x,y \in T$. We define 
\[
t(x,y) = ((x\ra y)\we(\s y \ra \s x)) \vee ((y\ra x)\we(\s x \ra \s y)).
\] 
We also define
\[
\SNAc = \SNA + \{\square(x\vee y) = \square(x) \vee \square(y)\} + \{ t(x,y) = 1\}.
\]
Let $\C$ be the class of totally ordered members of $\SNA$.

\begin{theorem}
$\V(\C) = \SNAc$.
\end{theorem}

\begin{proof}
A straightforward computation shows that $\C \subseteq \SNAc$, so $\V(\C) \subseteq \SNAc$.
In order to show the converse inclusion, let $T\in \SNAc$. If $T$ is trivial then $T\in \V(\C)$.
Now suppose that $T$ is not trivial. Let $\alpha:T\ra \prod_{P\in \oX(T)}T/P$ be the homomorphism defined by $\alpha (x)=(x/P)_{P\in \oX(T)}$. We will show that 
the homomorphism $\alpha$ is injective. Let $x,y\in T$ such that $\alpha(x) = \alpha(y)$.
Then $x/P = y/P$ for every $P\in \oX(T)$. Thus, $s(x,y) \in P$ for every $P\in \oX(T)$. 
It follows from Corollary \ref{cor-pf} that $s(x,y) = 1$. Hence, $x = y$.
We have proved that $\alpha$ is a monomorphism. Besides, it follows from Lemma \ref{lem-ch} that
$T/P$ is a chain for every $P\in \oX(T)$. Therefore, $T\in \V(\C)$. Then, $\SNAc\subseteq \V(\C)$.
Therefore, $\V(\C) =\SNAc$. 
\end{proof}

Note that $\SNAc$ is a proper subvariety of $\SNA$. Indeed, let $A$ be the subresiduated lattice given in Example \ref{ex1}.
We have that $\K(A)\in \SNA$. Consider $x = (a,0)$ and $y=(b,0)$, which are elements of $\K(A)$. Then
$\square x \vee \square y = (0,0)$ and $\square(x\vee y) = (1,0)$, so $\K(A) \notin \SNAc$.

\section{Centered subresiduated Nelson algebras} \label{s6}

A Kleene algebra (Nelson algebra) is called centered if there exists an element which
is a fixed point with respect to the involution, i.e., an element $\ce$ such that 
$\s \ce = \ce$. This element is necessarily unique. If $T = \K(A)$ where $A$ is a bounded distributive
lattice, the center is $\ce = (0,0)$.
Let $T$ be a centered Kleene algebra. We define the following condition:
\begin{center}
$\CK$ For every $x,y\in T$ if $x, y\geq \ce$ and $x\we y\leq \ce$
then there exists $z\in T$ such that $z\vee \ce = x$ and $\s z \vee \ce = y$.
\end{center}
The condition $\CK$ is not necessarily satisfied in every centered Kleene algebra, 
see for instance Figure 1 of \cite{CCSM}. However, every centered Nelson algebra
satisfies the condition $\CK$ (see \cite[Theorem 3.5]{C} and \cite[Proposition 6.1]{CCSM}).

The following two properties are well known: 
\begin{itemize}
\item Let $T$ be a Kleene algebra. Then $T$ is isomorphic to $\K(A)$ for some
bounded distributive lattice $A$ if and only if $T$ is centered and satisfies the
condition $\CK$ (see \cite[Theorem 2.3]{C} and \cite[Proposition 6.1]{CCSM}).
\item Let $T$ be a Nelson algebra. Then $T$ is isomorphic to $\K(A)$ for some
Heyting algebra $A$ if and only if $T$ is centered (see \cite[Theorem 3.7]{C}).
\end{itemize}

An algebra $\langle T,\we,\vee, \ra,\s,0,1,\ce \rangle$ is
a centered subresiduated Nelson algebra if $\langle T,\we,\vee, \ra,\s,0,1\rangle$ 
is as a subresiduated Nelson algebra and $\ce$ is a center.
We write $\SNAc$ for the variety whose members are centered subresiduated Nelson algebras.

In this section we prove that given $T\in \SNA$, $T$ is isomorphic to $\K(A)$
for some subresiduated lattice $A$ if and only if $T$ is centered and satisfies the
condition $\CK$ (we also show that the condition $\CK$ is not necessarily satisfied in every
centered subresiduated Nelson algebra). Finally we show that there exists a categorical equivalence
between $\SRL$ and the full subcategory of $\SNAc$ whose objects satisfy the condition $\CK$
\footnote{If $\mathrm{V}$ is a variety of algebras, with an abuse of notation we write it in this way to refer to the algebraic category associated to this variety.}. 

We start with some preliminary results.

\begin{lemma}\label{ce1}
Let $T\in \SNAc$ and $x, y\in T$. Then the following conditions are satisfied:
\begin{enumerate}[\normalfont 1)]
\item $\ce \ra x = 1$.
\item $((x\we y)\vee \ce)\ra 0 = (x\we y)\ra 0$.
\item  $(x\we y)\ra 0 = 1$ if and only if $x\we y \leq \ce$.
\end{enumerate}
\end{lemma}

\begin{proof}
In order to show 1), first we will see that $\ce \ra 0 = 1$. 
Indeed, $\ce \ra 0 = (\ce \we \sim \ce) \ra \sim (\ce \ra \ce) = 1$,
so $c\ra 0 = 1$. Taking into account that $\ce\ra 0 = 1$ and $0\ra x = 1$,
it follows from Proposition \ref{pvi} that $\ce \ra x = 1$. The condition 2)
is a direct consequence of 1).

In order to prove 3), suppose that that
$(x\we y)\ra 0 = 1$, so 
\[
x\we y = (x\we y)\we ((x\we y)\ra 0) \leq (x\we y) \we \s(x\we y)\leq \ce.
\]
Hence,  $x\we y \leq \ce$. The converse follows from 1).
\end{proof}

Let $T\in \SNAc$. We define the following condition:
\begin{center}
$\CC$ For every $x, y\in T$, if $(x\we y)\ra 0 = 1$ then there exists
$z\in T$ such that $z\vee \ce = x\vee \ce$ and $\s z \vee \ce = y\vee \ce$.
\end{center}

\begin{lemma}\label{ce2}
Let $T\in \SNAc$.
Then $\rho$ is surjective if and only if $T$ satisfies the condition $\CC$.
\end{lemma}

\begin{proof}
Suppose that $\rho$ is surjective. Let $x, y \in T$ such that $(x\we y)\ra 0 = 1$,
i.e., $x/\theta \we y/\theta = 0/\theta$. Since $\rho$ is surjective, there exists $z\in T$
such that $z/\theta = x/\theta$ and $\s z/\theta = y/\theta$, i.e., 
$z\ra x = 1$, $x\ra z = 1$, $\s z \ra y = 1$, $y \ra \s z = 1$. We will show that
$x\vee \ce \leq z\vee \ce$. It follows from Proposition \ref{pvi} 
that $(x\vee \ce) \ra (z\vee \ce) = 1$. Thus,
\[
x\vee \ce = (x\vee \ce)\we ((x\vee \ce) \ra (z\vee \ce)) \leq (x\vee \ce)\we (\s(x\vee \ce) \vee (z\vee c)) = (x\vee \ce)\we (z\vee \ce),
\]
so $x\vee \ce \leq z\vee \ce$. The same argument shows that $z\vee \ce \leq x\vee \ce$, so $x\vee \ce = z\vee \ce$.
Similarly, we get $\s z\vee \ce = y\vee \ce$. Hence, condition $\CC$ is satisfied.

Conversely, suppose that $\CC$ is satisfied and let $x,y \in T$ be such that 
$x/\theta \we y/\theta = 0/\theta$, i.e., $(x\we y)\ra 0 = 1$. It follows from
hypothesis that there exists $z\in T$ such that $z\vee \ce = x\vee \ce$ and 
$\s z \vee \ce = y\vee \ce$. Then $(x\vee \ce) \ra x = (z\vee \ce)\ra x$,
i.e., $(z\vee \ce)\ra x = 1$. Besides, $(z\vee \ce)\ra x = (z\ra x)\we (\ce \ra x)$,
and it follows from Lemma \ref{ce1} that $\ce \ra x = 1$, so $z\ra x = 1$.
The same argument shows that $x\ra z = 1$, so $x/\theta = z/\theta$. 
For the same reason, we get $\s z/\theta = y/\theta$. Therefore, $\rho$ is surjective. 
\end{proof}

\begin{lemma}\label{ce3}
Let $T\in \SNAc$. Conditions $\CC$ and $\CK$ are equivalent.
\end{lemma}

\begin{proof}
Suppose that $T$ satisfies $\CK$ and let $x, y\in T$ such that $(x\we y)\ra 0 = 1$.
It follows from Lemma \ref{ce1} that $x\we y \leq \ce$. Let $\hat{x} = x\vee \ce$ and
$\hat{y} = y\vee \ce$. Then $\hat{x}, \hat{y}\geq \ce$ and $\hat{x}\we \hat{y} \leq \ce$,
so it follows from hypothesis that there exists $z\in T$ such that $z\vee \ce = \hat{x}$ 
and $\s z \vee \ce = \hat{y}$, i.e., $z\vee \ce = x\vee \ce$ and $\s z \vee \ce = y\vee \ce$.

Conversely, suppose that $\CC$ is satisfied. Let $x,y\in T$ be such that 
$x, y\geq \ce$ and $x\we y\leq \ce$ It follows from Lemma \ref{ce1} that
$(x\we y)\ra 0 = 1$. Thus, it follows from hypothesis that there exists
$z\in T$ such that $z\vee \ce = x\vee \ce$ and $\s z \vee \ce = y\vee \ce$, i.e.,
$z\vee \ce = x$ and $\s z \vee \ce = y$.
\end{proof}

\begin{theorem}
Let $T\in \SNA$. Then $T$ is isomorphic to $\K(A)$ for some subresiduated lattice $A$
if and only if $T$ has center and it satisfies the condition $\CK$.
\end{theorem}

\begin{proof}
Suppose that $T \cong \K(A)$ for some $A\in \SRL$. Let $x, y \in \K(A)$ such that
$x, y \geq \ce$ and $x\we y \leq \ce$. Thus, there exists $(a,b) \in \K(A)$ such
that $x = (a,0)$ and $y = (b,0)$. The element $z = (a,b) \in \K(A)$ satisfies that  
$z\vee \ce = x$ and $\s z \vee \ce = y$. Hence, $\K(A)$ satisfies $\CK$. Since this condition
is preserved by isomorphisms then $T$ satisfies $\CK$. Furthermore, the fact that $\K(A)$ has got a center,
which is $(0,0)$, implies that $T$ has a center. The converse follows from Lemma \ref{ce2},
Lemma \ref{ce3} and Theorem \ref{rept}.
\end{proof}

Let $A$ be the subresiduated lattice given in Example \ref{ex1}. Then $\K(A)\in \SNAc$.
Let $T$ be a subset of $\K(A)$ given by the following Hasse diagram, which is 
the Hasse diagram of Figure 1 of \cite{CCSM}: 
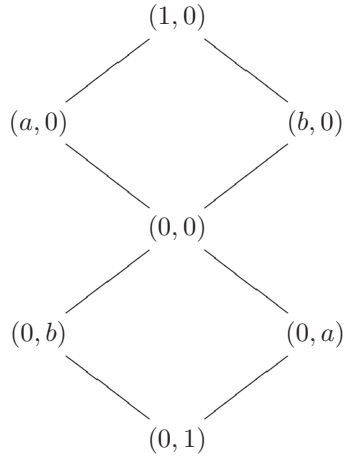
\begin{figure}[h]
    \centering
   \[
\xymatrix{
   &  (1,0)\\
\ar@{-}[ur] (a,0) &  & \ar@{-}[ul] (b,0)\\
   & \ar@{-}[ul] (0,0) \ar@{-}[ur]\\
   \ar@{-}[ur]  (0,b) &  & \ar@{-}[ul] (0,a)\\
   & \ar@{-}[ul] (0,1) \ar@{-}[ur]&\\
}
\]
    \caption{A centered subresiduated Nelson algebra that does not meet Condition $\CK$.}
    \label{fig:Hasse}
\end{figure}

Since $T$ is a subalgebra of $\K(A)$, we get $T\in \SNAc$. 
Note that we have $(a,0), (b,0) \geq (0,0)$ and $(a,0)\we (b,0) =  (0,0)$. However there is
not $z\in T$ such that $z\vee (0,0) = (a,0)$ and ${\sim} z \vee (0,0) = (b,0)$.
Therefore, $T$ does not satisfy $\CK$. It is also interesting to note that $U = \{(0,1),(1,0)\} \in \SNA$
but $U$ does not have a center. 
\vspace{1pt}

If $A\in \SRL$ then $\K(A) \in \SNA$. 
Besides, if $f:A\ra B$ is a morphism in $\SRL$ then 
it follows from a direct computation that the map $\K(f): \K(A) \ra \K(B)$ given by $\K(f)(a,b): = (f(a),f(b))$
is a morphism in $\SNA$. Moreover, $\K$ can be extended to a functor from $\SRL$ to $\SNA$.  
Conversely, if $T\in \SNA$ then $\C(T) =: T/\theta \in \SRL$. If $f:T\ra U$ is a morphism in $\SNA$ then
$\C(f):T/\theta \ra U/\theta$ given by $\C(f)(x/\theta) = f(x)/\theta$ is a morphism in $\SRL$. 
Moreover, $\C$ can be extended to a functor from $\SNA$ to $\SRL$.

\begin{lemma} \label{alpha}
Let $A\in \SRL$. Then the map $\alpha_{A}:A \ra \C(\K(A))$ given by $\alpha_{A}(a) = (a,\neg a)/\theta$
is an isomorphism.
\end{lemma}

\begin{proof}
We write $\alpha$ in place of $\alpha_{A}$.
First we will show that $\alpha$ is a well defined map. Let $a\in A$. Then $a\we \neg a = 0$. 
Thus, $(a,\neg a)/\theta \in \K(A)/\theta$.
Let $a\in A$. Then $(a,\neg a) /\theta = (a,0)/\theta$. 
It is immediate that $\alpha$ is a homomorphism. 
The injectivity of $\alpha$ is also immediate. 
In order to show that $\alpha$ is suryective, let $y\in \C(\K(A))$, so $y = (a,b)/\theta$
for some $a, b\in A$ such that $a \we b = 0$. Moreover, $y = (a,\neg a)/\theta$,
so $y = \alpha(a)$. Thus, $\alpha$ is suryective. Therefore, $\alpha$ is an isomorphism.
\end{proof}

A direct computation shows that if $f:A \ra B$ is a morphism in $\SRL$ and $a\in A$
then $\C(K(f))(\varphi_{A}(a)) = \varphi_{B}(f(a))$, and that if $f:T\ra U$ is a morphism in $\SNA$ and $x\in T$
then $\K(\C(f))(\rho_{T}(x)) = \rho_{U}(f(x))$. 

\begin{remark}
Note that if $T, U$ are Kleene algebras, $T$ has center and $f$ is a morphism in $\SNA$ from $T$ to $U$, then 
$U$ has center and $f(\ce) = \ce$. Indeed, $f(\ce) = f(\sim \ce) = \sim f(\ce)$.
\end{remark}

Therefore, the following result follows from the previous results of this section, and Lemmas \ref{ce2} and \ref{ce3}.

\begin{theorem}\label{center}
There exists a categorical equivalence between $\SRL$ and $\SNAc$.
\end{theorem}

\section{Conclusions and two open problems}

In this paper we extended, in the framework of sr-lattices, the well known twist construction given for Heyting algebras.
In order to make it possible, we introduced and studied the variety $\SNA$
by showing that every subresiduated Nelson algebra can be represented as a twist structure of an sr-lattice. 
We also characterized the congruences of subresiduated Nelson algebras and we applied this result in order to obtain
some additional properties for these algebras. In particular, we described simple and subdirectly
irreducible algebras, we proved that $\SNA$ has EDPC and CEP, we presented an equational
base for the variety of $\SNA$ generated by the class of its totally ordered members and finally we proved
that there exists a categorical equivalence between $\SRL$ and $\SNAc$.

We finish this paper by considering two (open) problems concerning the matter of this paper. 
 
\subsection*{Problem 1: Generalize the term equivalence between Nelson algebras and Nelson lattices}

We assume the reader is familiar with commutative residuated
lattices \cite{Ts}. An involutive residuated lattice is a
bounded, integral and commutative residuated lattice $(T,\we,
\vee, \ast,\ra, 0, 1)$ such that for every $x\in T$ it holds that
$\neg \neg x = x$, where $\neg x: = x\ra 0$ and $0$ is the first
element of $T$ \cite{BC}. In an involutive residuated lattice it
holds that $x \ast y = \neg (x \ra \neg y)$ and
$x\ra y = \neg (x \ast \neg y)$.
A Nelson lattice \cite{BC} is an involutive residuated
lattice $(T,\we,\vee, *,\ra,0,1)$ which satisfies the additional
inequality
$(x^2 \ra y)\we ((\neg y)^2 \ra \neg x) \leq x\ra y$,
where $x^2:=x\ast x$. See also \cite{V}.

\begin{remark} \label{br3}
Let $(T,\we, \vee, \Rightarrow,{\sim}, 0,1)$ be a Nelson algebra.
We define on $T$ the binary operations $*$ and $\ra$ by
$x*y:=  \sim (x \Rightarrow {\sim} y) \vee {\sim} (y \Rightarrow {\sim} x)$
and $x \ra y :=  (x \Rightarrow y) \we ({\sim} y\Rightarrow {\sim} x)$.
Then \cite[Theorem 3.1]{BC} says that $(T,\we, \vee, \ra,*, 0,1)$
is a Nelson lattice. Moreover, ${\sim} x = \neg x = x\ra 0$.

Let $(T,\we,\vee,*,\ra,0,1)$ be a Nelson lattice. We define on $T$
a binary operation $\Rightarrow$ and a unary operation $\sim$ by
$x \Rightarrow y:= x^2 \ra y$ and
$\sim x:= \neg x$,
where $x^2 = x*x$. Then Theorem 3.6 of \cite{BC} says that the
$(T,\we, \vee,\Rightarrow,{\sim},0,1)$ is a Nelson
algebra. In \cite[Theorem 3.11]{BC} it was also proved that the
category of Nelson algebras and the category of Nelson lattices
are isomorphic. Taking into account the construction of this
isomorphism in that paper we have that the variety of Nelson
algebras and the variety of Nelson lattices are term equivalent
and the term equivalence is given by the operations we have
defined before.
\end{remark}

\begin{remark} \label{imp}
Let $A$ be a Heyting algebra. Then
$(\K(A),\we,\vee,\Ra,{\sim},\ce,0,1)$ is a centered Nelson algebra. Thus,
it follows from Remark \ref{br3} that $\hat{\K}(A):= (\K(A),\we,\vee,*,\ra,\ce,0,1)$ is a centered Nelson lattice, 
where for $(a,b)$ and $(d,e)$ in $\K(A)$ the operations
$\ast$ and $\ra$ are given by
\begin{eqnarray*}
   (a,b) * (c,d) =  (a\we c, (a\ra d)\we (c\ra b)),\\
   (a,b) \ra (c,d) =  ((a\ra c)\we (d\ra b), a \we d).
\end{eqnarray*}
\end{remark}

The following question naturally arises: 
\begin{itemize}
\item Is it there a variety of algebras, 
in the language of Nelson lattices, which is term equivalent to the variety of subresiduated Nelson lattices?
\end{itemize}
We do not have an answer for this question.
 
In \cite[Corollary 4.18]{CCSM} it was proved that there exists an equivalence between $\SRL$ and an algebraic category
whose objects are in the language of centered Nelson lattices. We write $\KSRL$ for this algebraic category. 
In particular, if $A\in \SRL$ then $\hat{\K}(A)\in \KSRL$,
where the binary operation $\ra$ is defined as in Remark \ref{imp}. Moreover, for every $T\in \KSRL$ there exists
$A\in \SRL$ such that $T$ and $\hat{\K}(A)$ are isomorphic algebras.

The following result follows from Theorem \ref{center} and \cite[Corollary 4.18]{CCSM}.

\begin{proposition} \label{equivalence}
The categories  $\KSRL$ and $\SNAc$ are equivalent. 
\end{proposition}

The following question also naturally arises: 
\begin{itemize}
\item Is there a variety of algebras, 
in the language of centered Nelson lattices, which is term equivalent to the variety of centered subresiduated Nelson lattices?
\end{itemize}
We do not have an answer for this question. However, we know that the usual construction (the one used to show the term equivalence between  Nelson algebras and Nelson lattices) does not work by considering centered subresiduated Nelson algebras and the objects of the algebraic category
$\KSRL$. In order to show it, assume that the construction works,
which is equivalent to say that this works for the centered subresiduated Nelson algebra $(\K(A),\we,\vee,\Ra,{\sim},\ce,0,1)$
and the algebra $\hat{\K}(A)$ where $A$ is an arbitrary sr-lattice.
Then, for every $A\in \SRL$ and $a, b, c, d\in A$
such that $a\we b = c\we d = 0$ the equality $(a,b) \Ra (c,d) = (a,b)^2 \ra (c,d)$ is satisfied,
i.e., the inequality $a\ra c \leq d \ra (a\ra b)$ is satisfied.
Consider the Boolean algebra of four elements, where $a$ and $b$ are the atoms, and $D = \{0,a,1\}$.
We have that $(A,D)$, or directly $A$, is an sr-lattice. Define $c=a$ and $d=b$. We have that
$a\ra c = 1$ and $d\ra (a\ra b) = d\ra 0 = a$, so $a\ra c \nleq d \ra (a\ra b)$, which is a contradiction.

\subsection*{Problem 2: Generalize the equivalence between $\NA$ and a
category of enriched Heyting algebras}

We know that every Nelson algebra can be represented as a special twist structure of a Heyting
algebra. This correspondence was formulated as a categorical equivalence (by Sendlewski
in the early 1990's, see \cite{S2} and also in \cite{Vig}) between Nelson algebras and a category of enriched Heyting algebras,
which made it possible to transfer a number of fundamental results from the more widely
studied theory of intuitionistic algebras to the realm of Nelson algebras. 

The objects of the category of enriched Heyting algebras above mentioned are pairs $(A,R)$, where $A$ is a Heyting
algebra and $R$ is a Boolean congruence of $A$ (i.e., $R$ is a congruence such that $A/R$ is a Boolean algebra). 
Congruences of any Heyting algebra can be represented by filters, and filters corresponding to Boolean congruences 
are precisely those containing all dense elements, i.e., elements a such that $\neg a = 0$ \cite{RS}. This
allows to replace the notion of a Boolean congruence by the notion of a filter
containing dense elements, which will be called Boolean filter. 
Thus, we may consider pairs $(A,F)$ where $A$ is a Heyting algebra and $F$ is a Boolean filter.
The categorical equivalence for $\NA$ can be presented as follows. We define $\Hs$ as the category whose objects are pairs $(A,F)$, where $A$ is a Heyting algebra and $F$ is a Boolean filter, and whose morphisms
$f:(A,F)\ra (B,G)$ are homomorphisms such that $f(F)\subseteq G$. If $(A,F) \in \Hs$ then 
\[
\K(A,F): = \{(a,b)\in A\times A: a\we b = 0\;\text{and}\;a\vee b \in F\}
\]
is a Nelson algebra with the operations mentioned in Section \ref{s2}. If $f:(A,F)\ra (B,G)$ is a
morphism in $\Hs$ then $\K(f):\K(A,F) \ra \K(B,G)$ defined by $\K(f)(a,b):= (f(a),f(b))$ is a morphism in $\NA$.
Moreover, $\K$ can be extended to a functor from $\Hs$ to $\NA$.
Conversely, if $T\in \NA$ then $\C(T):= (T/\theta, T^{+}/\theta) \in \Hs$, where $T^+:= \{x\in T:x\geq \sim x\}$. 
If $f:T\ra U$ is a morphism in $\NA$ then $\C(f):\C(T)\ra \C(U)$ given by $\C(f)(x/\theta): = f(x)/\theta$ is a morphism in $\Hs$.
Moreover, $\C$ can be extended to a functor from $\NA$ to $\Hs$. Furthermore, if $(A,F) \in \Hs$ then
$\alpha:(A,F) \ra \C((\K(A,F))$, given as in Lemma \ref{alpha}, is an isomorphism in $\Hs$, and if $T\in \NA$
then $\rho:T \ra \K(\C(T))$, given as in Theorem \ref{rept}, is an isomorphism in $\NA$. The functors $\K$ and $\C$ establish
a categorical equivalence between $\Hs$ and $\NA$ \cite{V,Vig}. Therefore, it is natural to think in
the following problem:
\begin{itemize}
\item Is it possible to generalize the above mentioned categorical equivalence in the framework of $\SNA$?
\end{itemize}

Let $A\in \SRL$. One might think that a possible solution is to consider filters of an algebra $A\in \SRL$ that contain the set of dense elements $\De(A):=\{a\in A: \neg a=0\}$ (as in the case of Heyting algebras). It is known that in Heyting algebras
the set of dense elements is a filter. However, in sr-lattices it is not necessarily 
satisfied. Indeed, let $A$ be the sr-lattice given in Example \ref{ex1}.
Then $\De(A) = \{a,b,1\}$, which is not a filter. In sr-lattices the congruences are given by open filters, i.e., filters $F$ such that $1\ra a\in F$
whenever $a\in F$ \cite{EH,CJ}. 
In order to try to generalize the equivalence between $\Hs$ and $\NA$, it seems natural to work with pairs
$(A,F)$, where $A$ is a sr-lattice and $F$ is an open filter such that $\De(A)\subseteq F$
(these filters will be called subresiduated filters).

\begin{lemma} \label{dense}
Let $A\in \SRL$. Then $\De(A)=\{a\vee\neg a : a\in A\}$.
\end{lemma}

\begin{proof}
Let $a\in \De(A)$. Then $\neg a=0$, so $a = a\vee\neg a$.
Hence $a\in \{b\vee\neg b : b\in A\}$.
Conversely, let $a\in A$. We shall see that $\neg (a\vee \neg a)=0$. 
Indeed, $\neg (a\vee\neg a)= \neg a \we \neg \neg a = 0$. Hence, $a \vee \neg a\in \De(A)$.
\end{proof}

We write $\SRLs$ for the set whose elements are pairs $(A,F)$, with $A\in \SRL$ and
$F$ a subresiduated filter.
\vspace{1pt}

Let $(A,F)\in \SRLs$. We define 
\[
\K(A,F): = \{(a,b)\in A\times A: a\we b = 0\;\text{and}\;a\vee b \in F\}.
\]
Note that $\K(A,A) = \K(A)$.
Let $(a,b), (c,d) \in \K(A,F)$. Then, we have that $(a,b)\we (c,d), (a,b)\vee (c,d)$ and $\sim (a,b) \in \K(A,F)$
(for details see for instance \cite{Vig}).

\begin{lemma}\label{lem-eq0}
Let $(A,F)\in \SRLs$ and $(a,b), (c,d)\in \K(A,F)$.
Then $(a,b)\Rightarrow (c,d)\in \K(A,F)$.
\end{lemma}

\begin{proof}
Let $(a,b), (c,d)\in \K(A,F)$. We will prove that $(a \ra c, a\we d)\in\K(A,F)$. First, note that since $(a \ra c, a\we d)\in\K(A)$, $(a \ra c)\we a\we d=0$. We will see that $(a \ra c)\vee (a\we d)\in F$. By distributive law, $(a \ra c)\vee (a\we d)=((a\ra c)\vee a)\we ((a\ra c)\vee d)$. It is immediate to see that $\neg a \leq a \ra c$ and in consequence $\neg a\vee a\leq (a\ra c)\vee a$. Since $\neg a\vee a\in \De(A)\subseteq F$, we get $(a\ra c)\vee a\in F$. We only need to check that $(a\ra c)\vee d\in F$. To do so, we know that $(1\ra (c\vee d))\we ((c\vee d)\ra c)\leq 1\ra c$. Since, $ (c\vee d)\ra c=(c \ra c)\we (d\ra c)=d \ra c$, and $d \ra c= (d\ra c)\we (d \ra d)= d \ra (c\we d)=d \ra 0$, we get $(1\ra (c\vee d))\we \neg d\leq 1\ra c$. Then,
$((1\ra (c\vee d))\we \neg d)\vee d\leq (1\ra c)\vee d$, i.e.,
\[
((1\ra (c\vee d))\vee d)\we (\neg d\vee d) \leq (1\ra c)\vee d.
\]
Since $c\vee d\in F$ and $F$ is an open filter, it follows that $1\ra (c\vee d)\in F$ and thus, $(1\ra (c\vee d))\vee d\in F$. From $\neg d\vee d\in F$, we get $(1\ra c)\vee d\in F$. It is easy to see that $(1\ra c)\vee d\leq (a\ra c)\vee d$, therefore we get $(a\ra c)\vee d\in F$, which was our aim.
\end{proof}

Thus, for $(A,F)\in \SRLs$ we have that $\langle \K(A,F),\we,\vee,\Ra,\s,(0,1),(1,0) \rangle$ is a subresiduated Nelson algebra because this is a subalgebra of the subresiduated Nelson algebra $\langle \K(A),\we,\vee,\Ra,\s,(0,1),(1,0) \rangle$. However, this construction can not be extended to a categorical equivalence, unless following the construction which shows the categorical equivalence for the category of enriched Heyting algebras we have mentioned in this section. Indeed, let us consider $T$ as the subresiduated Nelson algebra whose Hasse diagram is given in Figure \ref{fig:Hasse}. Then $T/\theta$ is isomorphic to the sr-lattice $A$ given in the Example \ref{ex1}. Note that $\De(A)=\{a,b,1\}$. So, the only open filter $F$ that contains the set $\De(T/\theta)$ is $F=T/\theta $ and we know that $T$ is not isomorphic to $\K(T/\theta,T/\theta) = \K(T/\theta)$.

\subsection*{Acknowledgments}

This work was supported by Consejo Nacional de Investigaciones Cient\'ificas y T\'ecnicas
(PIP 11220170100195CO and PIP 11220200100912CO, CONICET-Argentina),
Universidad Nacional de La Plata (11X/921) and Agencia Nacional de Promoci\'on Cient\'ifica y
Tecnol\'ogica (PICT2019-2019-00882, ANPCyT-Argentina) and the National Science Center (Poland), grant number \linebreak 2020/39/B/HS1/00216, ``Logico-philosophical foundations of geometry and \linebreak topology'' and the MOSAIC project. This last project has received funding from the European Union’s Horizon 2020 research and innovation programme under the Marie Skłodowska-Curie grant agreement No 101007627.


\newpage

-----------------------------------------------------------------------------------
\\
Noem\'i Lubomirsky,\\
Centro de Matem\'atica de La Plata (CMaLP), \\
Facultad de Ciencias Exactas (UNLP), \\
and CONICET.\\
Casilla de correos 172,\\
La Plata (1900), Argentina.\\
nlubomirsky@mate.unlp.edu.ar

------------------------------------------------------------------------------------
\\
Paula Menchón,\\Department of Logic,\\
Faculty of Philosophy and Social Sciences,\\
Nicolaus Copernicus University in
Toruń, Poland, \\
Orcid: 0000-0002-9395-107X\\
Email address: paula.menchon@v.umk.pl

------------------------------------------------------------------------------------
\\
Hern\'an Javier San Mart\'in,\\
Centro de Matem\'atica de La Plata (CMaLP), \\
Facultad de Ciencias Exactas (UNLP), \\
and CONICET.\\
Casilla de correos 172,\\
La Plata (1900), Argentina.\\
hsanmartin@mate.unlp.edu.ar

\end{document}